\documentclass[a4paper, leqno]{article}
\usepackage[english]{babel}
\usepackage{amsmath,amscd,amssymb,amsthm}
\usepackage{physics}
\usepackage{pb-diagram}
\usepackage{cancel}

\newtheorem{definition}{Definition}[section]
\newtheorem{theorem}{Theorem}[section]
\newtheorem{lemma}{Lemma}[section]
\newtheorem{proposition}{Proposition}[section]
\newtheorem{corollary}[theorem]{Corollary}
\newtheorem{remark}{Remark}[section]

\begin{document}
\title{Structure of the Kuranishi Spaces of pairs of K\"ahler manifolds and Polystable Higgs bundles}
\author{Takashi Ono\thanks{Department of Mathematics, Graduate School of Science, Osaka University, Osaka, Japan, u708091f@ecs.osaka-u.ac.jp}}
\date{}
\maketitle
\begin{abstract}
Let $X$ be a compact K\"ahler manifold and $(E,\overline\partial_E,\theta)$ be a Higgs bundle over it. We study the structure of the Kuranishi space for the pair $(X, E,\theta)$ when the Higgs bundle admits a harmonic metric or equivalently when the Higgs bundle is polystable and the Chern classes are 0. Under such assumptions, we show that the Kuranishi space of the pair  $(X,E,\theta)$ is isomorphic to the direct product of the Kuranishi space of $(E,\theta)$ and the Kuranishi space of $X$. Moreover, when $X$ is a Riemann surface and  $(E,\overline\partial_E,\theta)$ is stable and the degree is 0, we show that the deformation of the pair  $(X,E,\theta)$ is unobstructed and calculate the dimension of the Kuranishi space. 
\end{abstract}
\section{Introduction}
Let $X$ be a complex manifold. A Higgs bundle $(E,\overline\partial_{E},\theta)$ over $X$ is a pair such that $(i)$ $(E,\overline\partial_E)$ is a holomorphic bundle over $X$, $(ii)$ $\theta$ is a holomorphic 1-form which takes value in $\mathrm{End}E$ and $\theta\wedge\theta=0$. We call a pair $X$ and $(E,\overline\partial_E,\theta)$ a \textit{holomoprhic-Higgs pair}.  In our previous paper \cite{O}, we studied the deformation problem of holomorphic-Higgs pairs differential geometrically: we studied the deformation problem when $X$ and $(E,\overline\partial_{E},\theta)$ deform simultaneously. We constructed the DGLA $(L,[,]_L, d_L)$ which governs the deformation differential geometrically, and constructed the Kuranishi space $Kur_{(X, E,\theta)}$. The Kuranishi space is an analytic space such that it contains all information of small deformations of the given holomorphic-Higgs pair. See section \ref{DGLA} for the details of the DGLA  $(L,[,]_L, d_L)$. \par
There are many interesting works on the deformation of pairs. The deformation problem of pairs $(X, E)$ where $X$ is a complex manifold and $E$ is a holomorphic bundle over $X$ was studied in \cite{CS, H}. In those works, the DGLA  and the Kuranishi space were constructed. When $X$ is a smooth projective variety and $E$ is a coherent sheaf on it, the deformation problem of pair was studied in \cite{IM}. The deformation of the holomorphic-Higgs pair was studied algebraically in \cite{Ma} and the controlling DGLA was obtained. The difference between our previous work \cite{O} and her work \cite{Ma} is the way to construct the DGLA. In \cite{Ma}, it was constructed purely algebraically, and on the other hand, in \cite{O}, it was constructed differential geometrically.\par
In this paper, we study the structure of the Kuranishi space. We study the structure of the Kuranishi space when $X$ is a compact K\"ahler manifold and the Higgs bundle $(E,\overline\partial_{E},\theta)$ admits a \textit{harmonic metric} $K$. From now on, we assume $X$ to be a compact K\"ahler manifold. \par
Before we recall the notion of the harmonic metrics, we state the main result of this paper. 
 Let $Kur_X$ be the Kuranishi space of $X$, $Kur_{(E,\theta)}$ be the Kuranishi space of the Higgs bundle $(E,\theta)$, and $Kur_{(E, D)}$ be the Kuranishi space of the flat bundle $(E, D).$ The flat bundle $D$ is obtained from the Higgs bundle and the harmonic metric.
  \begin{theorem}[Theorem \ref{main}]\label{1}
  Let $(X,\omega)$ be a compact K\"ahler manifold, $(E,\overline\partial_E,\theta)$ be a Higgs bundle over $X$ and, $K$ be a harmonic metric. Then
\begin{align*}
(Kur_{(X,E,\theta)},0)&\simeq(Kur_{(E,\theta)}\times Kur_X,0),\\
(Kur_{(X,E,\theta)},0)&\simeq(Kur_{(E,D)}\times Kur_X,0)
\end{align*}
holds as germs of analytic spaces.
\end{theorem}
As we explain afterward, the existence of harmonic metrics is related to the stability of the Higgs bundles. Theorem \ref{1} predicts that once we construct a moduli space of a pair of a compact K\"ahler manifold and a polystable Higgs bundle with vanishing Chern classes, such moduli space should locally decompose to the direct product of the Kuranishi space of the manifold and the Kuranishi space of the Higgs bundle which we cannot expect globally. The moduli space of pairs of K\"ahler manifold and stable bundles was considered in \cite{H, ST}.  However the author couldn't find a work that deals with pairs of K\"ahler manifolds and stable Higgs bundles.\par 
We have some consequences from Theorem \ref{1} for specific cases. Let $M$ be a Riemann surface with genus $g \ge 2$ and $(E,\overline\partial_E,\theta)$ be a stable Higgs bundle of degree 0. Under these assumptions, each deformations of $M$ and $(E,\overline\partial_E,\theta)$ are unobstructed and the dimensions of  $Kur_X$ is  $3g-3$ and $Kur_{(E,\theta)}$  is $2+r^2(2g-2)$ \cite{MK, N}. Here $r$ is the rank of $E$.  The following is straightforward from Theorem \ref{1}.
\begin{corollary}[Corollary \ref{sub}]
Let $M$ be a Riemann surface with genus $g \ge 2$ and $(E,\overline\partial_E,\theta)$ be a stable Higgs bundle of degree 0. Then the deformation of pair $(M, E,\theta)$ is unobstructed. Moreover, $Kur_{(M,E,\theta)}$ is a complex manifold and its dimension is $g(2r^2+3)-2r^2-1$.
\end{corollary}
The Corollary predicts that the moduli space of a pair of Riemann surfaces and stable Higgs bundles of degree 0 is smooth in a stable locus and its dimension is $g(2r^2+3)-2r^2-1$.
 \par
 We now recall the harmonic metrics.
Let $K$ be a Hermitian metric of $E$ and $\partial_K$ be the (1,0)-part of the Chern connection with respect to $K$ and $\overline\partial_E$. Let $\theta^\dagger_K$ be the adjoint of $\theta$ with respect to $K$. Let $D'_K:=\partial_K+\theta^\dagger_K$ and $D'':=\overline\partial_E+\theta$. We obtain a connection $D:=D'_K+D''$. We call a metric $K$ a harmonic metric when $X$ is a compact K\"ahler manifold and  $D$ is a flat connection (i.e. $D^2=0$). The existence of the harmonic metric for a Higgs bundle is related to the stability of it. Actually, we have the following. 
\begin{theorem}[\cite{N,S1}]
Let $(E,\overline\partial_E,\theta)$ be a Higgs bundle on $X$. $(E,\overline\partial_E,\theta)$ admits a harmonic metric $K$ if and only if it is polystable and $c_1(E)=c_2(E)=0$. Here, $c_i(E)$ is the i-th Chern class. Let $K_1$ and $K_2$ be harmonic metrics. Then there exists a decomposition $(E,\overline\partial_E,\theta)=\oplus_i(E_i,\overline\partial_{E_i},\theta_i)$ such that (i) the decomposition is orthogonal with respect to both $K_1$ and $K_2$, (ii) there exists $a_i>0$ such that  $K_1=a_iK_2$ on $(E_i,\overline\partial_{E_i},\theta_i).$
\end{theorem}
Hence, in other words, we study the Kuranishi space $Kur_{(X,E,\theta)}$ when the Higgs bundle is polystable and its Chern classes are 0. \par
We prove Theorem \ref{1} by showing certain DGLAs are quasi-isomorphic. We use the results of \cite{GM2}. Let $(L,[,],d_L)$ be a DGLA. When it has a suitable topological structure and has finite-dimensional cohomology in degrees 0 and 1 we call it \textit{analytic DGLA}. For an analytic DGLA $(L,[,],d_L)$, there exists an analytic space $Kur_L$ such that the germ of the analytic space $(Kur_L,0)$ represents a certain functor which comes from the DGLA. See section \ref{H Kur} for details. \par
 We note that the construction of $Kur_L$ is based on \cite{Ku}. The DGLAs that appear in this paper are analytic DGLAs by the standard Sobolev norms. The $Kur_L$ of such DGLAs are the standard versal deformation space.\par
It was shown in \cite{GM2} that, the germ $(Kur_L,0)$ is homotopy invariant: when two analytic DGLAs $(L_1,[,]_1,d_{L_1})$ and $(L_2,[,]_2,d_{L_2})$ are quasi-isomorphic, then the germs
$(Kur_{L_1},0)$ and $(Kur_{L_2},0)$ are analytic isomorphic. \par
We show that when $X$ is a K\"ahler manifold and $(E,\overline\partial_{E},\theta)$ has a harmonic metric $K$, or equivalently when it is polystable and the Chern classes are 0, then  $(L,[,]_L, d_L)$ is quasi-isomorphic to certain DGLAs. From $X$, we obtain the Kodaira-Spencer algebra $(A^*(TX),[,]_{SN},\overline\partial_{TX})$ which is a DGLA, and from the harmonic bundle $(E,\overline\partial_{E},\theta, K)$ we obtain two DGLAs $(A^*(\mathrm{End}E),[,]_{\mathrm{End}E},D)$ and $(A^*(\mathrm{End}E),[,]_{\mathrm{End}E},D'')$.  The $[,]_{SN}$ is the Schouten-Nijenhuys bracket and  $(L,[,]_L, d_L)$ is the DGLA which governs the deformation of pair. We first prove the following Theorem.
\begin{theorem}[Theorem \ref{q iso}]\label{th 1}
$(L,[,]_L,d_L)$ is quasi-isomorphic to $(A^*(\mathrm{End}E),[,]_{\mathrm{End}E},D)\oplus(A^*(TX),[,]_{SN},\overline\partial_{TX})$.
\end{theorem}
The DG-vector spaces that come from a harmonic bundle satisfy the formality condition. The following Corollary is straightforward from the formality.
\begin{corollary}[Corollary \ref{q iso3}]\label{co 1}
$(L,[,]_L,d_L)$ is quasi-isomorphic to $(A^*(\mathrm{End}E),[,]_{\mathrm{End}E},D'')\oplus(A^*(TX),[,]_{SN},\overline\partial_{TX})$.
\end{corollary}
 Theorem \ref{1} is obtained from the homotopy invariance of the germ of the Kuranishi spaces, Theorem \ref{th 1} and, Corollary \ref{co 1}.\par
The study of the structure of the Kuranishi space using the theory of DGLA was initiated in \cite{GM1}. They showed the moduli space of representation of the fundamental group for a compact K\"ahler manifold has quadratic singularities when the representation satisfies certain properties. They proved it by showing certain DGLAs are formal. The formality of some DGLAs is applied to some deformation problems \cite{BMM, BM}. In \cite{GM2}, the homotopy invariance of the Kuranishi space was proved. They applied it and studied the Kuranishi space of compact parallelizable nilmanifold.
   \subsubsection*{Acknowledgement}
   The author thanks his supervisor Hisashi Kasuya for his suggestion of this paper. The author also thanks for his patience, advice, and encouragement.
   The author was supported by JST SPRING, Grant Number JPMJSP2138.
      \section{Homotopy invariance of the Kuranishi Space}\label{H Kur}
\subsection{Differential graded Lie algerbas}
In this section, we review the notion of the Differential graded Lie algebra (DGLA for short). This section is based on \cite{M}.
\begin{definition}
A Differential-Graded vector space (DG vector space) is a pair $(L,d_L)$ such that $L=\oplus_iL^i$ is a $\mathbb{Z}$-graded vector space and $d:L\to L$ is a linear map such that $d(L^i)\subset L^{i+1}$ and $d\circ d=0.$
\end{definition}
Let $(L,d_L)$ be a DG vector space. A sub DG-vector space $(L'=\oplus_{i\in\mathbb{Z}}L'^{i},d_{L'})$ of $(L,d_L)$ is DG vector space such that for each $i$, $L'^i\subset L^i$ is a sub vector space and $d_{L'}$ is the restriction of $d_L$. \par
 A morphism $f:(L_1,d_{L_1})\to(L_2,d_{L_2})$ of DG vector spaces is a morphism of vector space $f: L_1\to L_2$ such that it commutes with the differentials. We note that  $f$ induces a morphism $H^i(f):H^i(L_1)\to H^i(L_2)$. Here $H^i(L_j)(j=1,2)$ is the $i$-th cohomology of $(L_j,d_{L_j})$. Let $(L, d_L)$ be a DG vector space and $(L',d_{L'})$ be a sub DG vector space of it. Then the inclusion of $L'^i$ to $L^i$ induces a morphism of DG vector space $i:(L',d_{L'})\to(L, d_L)$.
\begin{definition} A \emph{Differential graded Lie algebra} (DGLA)
    $(L,[,],d)$ is the data of a
    $\mathbb{Z}$-graded vector space
    $L=\oplus_{i\in Z}L^i$ with a bilinear bracket
    $[,]\colon L\times L\to L$ and a linear map $d\colon L\to L$ satisfying
    the following condition:\begin{enumerate}
    
    \item $[,]$ is homogeneous skewsymmetric:
     $[L^i,L^j]\subset L^{i+j}$ and
    $[a,b]+(-1)^{\overline{a}\overline{b}}[b,a]=0$ for every $a,b$ homogeneous.
    
    \item  Every triple of homogeneous elements $a,b,c$  satisfy the Jacobi identity
    \[ [a,[b,c]]=[[a,b],c]+(-1)^{\overline{a}\overline{b}}[b,[a,c]].\]
    
    \item  $d(L^i)\subset L^{i+1}$, $d\circ d=0$ and
    $d[a,b]=[da,b]+(-1)^{\overline{a}}[a,db]$ holds. The map $d$ is called the
    differential of $L$.
    \end{enumerate}
    \end{definition}
    \begin{definition}
    The Maurer-Cartan equation of a DGLA $(L,[,],d)$ is 
    \begin{equation*}
    da+\frac{1}{2}[a,a]=0, a \in L^1.
   \end{equation*}
   The solutions of the Maurer-Cartan equation are called Maurer-Cartan elements of the DGLA $(L,[,],d)$.
    \end{definition}
  Let $(L,[,],d_L)$ be a DGLA. We can consider $(L,d_L)$ as a Differential-Graded vector space (DG vector space for short). Let $(L_1,d_{L_1})$  and $(L_2,d_{L_2})$ be DG vector spaces.
  \begin{definition}
  Let $(L_1,[,]_{1},d_{L_1})$ and $(L_2,[,]_{2},d_{L_2})$ be DGLAs. A morphism $f:(L_1,[,]_{1},d_{L_1})\to(L_2,[,]_{2},d_{L_2})$ of DGLAs is a morphism of DG vector spaces $f:(L_1,d_{L_1})\to(L_2,d_{L_2})$ such that it commutes with brackets.
  \end{definition}
  Let $(L_1,d_{L_1})$  and $(L_2,d_{L_2})$ be DG vector spaces.  We say that  $(L_1,d_{L_1})$  and $(L_2,d_{L_2})$ are \textit{quasi-ismorphic} if there exists a morphism of DG vector space $f:(L_1,d_{L_1})\to(L_2,d_{L_2})$ such that $H^i(f)$ is an isomorphism for each $i$. \par
  Let $(L_1,[,]_1,d_{L_1})$ and $(L_2,[,]_2,d_{L_2})$ be DGLAs. We say that $(L_1,[,]_1,d_{L_1})$ and $(L_2,[,]_2,d_{L_2})$ are \textit{quasi-ismorphic} if there exists a family of 
  DGLA $\{(W_i,[,]_{W_i},d_{W_i})\}^n_{i=1}$ and a family of morphism of DGLA $\{f_i\}^{n+1}_{i=1}$ such that 
  \begin{equation*}
  L_1\xleftarrow{f_1}W_1\xrightarrow{f_2}W_2\xleftarrow{f_3}\cdots\xrightarrow{f_{n-1}}W_{n-1}\xleftarrow{f_{n}}W_n\xrightarrow{f_{n+1}}L_2
      \end{equation*}
      and each $f_i$ is a quasi-isomorphism of DG vector spaces.\par
      Let $(L,[,],d_L)$ be a DGLA. Since $d_L$ satisfies the Leibniz rule, $(H^{*}_L,[,],0)$ has the structure of DGLA.
      \begin{definition}
      Let $(L,[,],d_L)$ be a DGLA. $(L,[,],d_L)$ is called formal if it is quasi-isomorphic to $(H^*_L,[,],0)$.
      \end{definition}
      \subsection{Homotopy invariance of the Kuranishi Space}
     In this section, we review the homotopy invariance of the Kuranishi spaces based on \cite{GM1, GM2}.\par  Let $X$ be an analytic space and $x\in X$. We denote the germ of $X$ at $x$ as $(X,x)$. We denote by $\mathcal{O}_{(X,x)}$ the corresponding analytic local ring consisting of germs of functions on $X$ which are analytic at $x$. Let $A$ be a local ring. We denote the completion of $A$ with respect to its maximal ideal as $\widehat{A}$: the complete local ring of $(X,x)$ is $\widehat{\mathcal{O}}_{(X,x)}$.\par
   Let $\mathbb{K}$ be a field. Let $R$ be a local $\mathbb{K}$-algebra, $Art_{\mathbb{K}}$ be the category of Artin local $\mathbb{K}$-algebras and $Set$ be the category of sets. We have a naturally defined functor
    \begin{equation*}
    \mathrm{Hom}(R,\cdot): Art_{\mathbb{K}}\to Set
    \end{equation*}
    which we denote $F_R$. Let $F: Art_{\mathbb{K}}\to Set$ be a functor. We say that an analytic germs $(X,x)$ \textit{pro-represents} $F$ if $F$ and $F_{\widehat{\mathcal{O}}_{(X,x)}}$ are naturally isomorphic. In this paper, $\mathbb{K}$ is often $\mathbb{C}$. Using the results of \cite{A, Gu}, it was proved in \cite{GM1} such that the following four conditions  are equivalent:
          \begin{itemize}
      \item[(1)] The analytic germs of $(X,x)$ and $(Y,y)$ are analytic isomorphisc.
      \item[(2)] The analytic local rings $\mathcal{O}_{(X,x)}$ and $\mathcal{O}_{(Y,y)}$ are isomorphic.
      \item[(3)] The complete local rings  $\widehat{\mathcal{O}}_{(X,x)}$ and $\widehat{\mathcal{O}}_{(Y,y)}$ are isomorphic.
      \item[(4)] The functor $F_{\widehat{\mathcal{O}}_{(X,x)}}$ and $F_{\widehat{\mathcal{O}}_{(Y,y)}}$ are naturally isomorphic.
      \end{itemize} 
     
  Let $(L,[,],d_L)$ be a DGLA. Let $C^1(L)$ be the complement of the 1-coboundary $B^1(L)\subset L^1$. We define a functor $Y_L:Art_{\mathbb{K}}\to Set$ such that for $A\in
  Art_{\mathbb{K}}$
     \begin{equation*}
       Y_L(A)=\bigg\{\eta\in C^1(L)\otimes m_A: d\eta+\frac{1}{2}[\eta,\eta]=0\bigg\}.
      \end{equation*}
Here, $m_A$ is the maximal ideal of the Artin local $\mathbb{K}$-algebra $A$.  It was proved in $\cite{GM2}$ that $Y_L$ is \textit{pro-representable}: that is, there exists a complete local $\mathbb{K}$-algebra $R_L$ such that $Y_L$ and $F_{R_L} $  are naturally isomorphic.\par
Let $(L_i,[,]_i,d_{L_i})$ $(i=1,2)$ be DGLAs and $C^1(L_i)$ $(i=1,2)$ be the complement of the coboundaries $B^1(L_i)$. Let $f:(L_1,[,]_1,d_{L_1})\to(L_2,[,]_2,d_{L_2})$ be a morphism of DGLA. We assume that 
\begin{itemize}
\item[(i)] $H^1(f)$ is an isomorphism.
\item[(ii)] $H^2(f)$ is an injection.
\end{itemize}
Then it was proved in \cite{GM2} that, if a morphism $f:(L_1,[,]_1,d_{L_1})\to(L_2,[,]_2,d_{L_2})$ satisfies $(i)$ and $(ii)$, then $R_{L_1}$ and $R_{L_2}$ is isomorphic.
\par
We next introduce the notion of \textit{analytic DGLA}. A \textit{normed  DGLA} $(L,[,],d_L)$ is a  DGLA such that each $L^i$ is a normed vector space and with respect to the norms
\begin{itemize}
\item[(1)] $d_L: L^i\to L^{i+1}$ is continuous.
\item[(2)] $[,]:L^1\otimes L^1\to L^2$ is contionius.
\end{itemize}
We let $\widehat{L}^i$ to be the completion of $L^i$ with respect to the norm.\par
An analytic DGLA is a normed DGLA $(L,[,],d_L)$ such that it has finite-dimensional cohomology in degrees 0 and 1 and each $\widehat{L}^i$ has continuous splitting:
\begin{equation*}
0\to Z^j(\widehat{L})\to\widehat{L}^j\to B^{j+1}(\widehat{L})\to 0
\end{equation*}
and 
\begin{equation*}
0\to B^j(\widehat{L})\to Z^j(\widehat{L})\to H^i(\widehat{L})\to 0.
\end{equation*}
 It was proved in \cite[Section 2]{GM2} that for an analytic DGLA $(L,[,],d_L)$, there exists a germ of analytic space $(Kur_L,0)$ such that $F_{\mathcal{O}_{(Kur_L,0)}}$ is naturally isomorphic to $Y_L$. Therefore, $R_L$ is isomorphic to $\widehat{\mathcal{O}}_{(Kur_L,0)}$. See \cite[Chapter 3]{M} for more details for the functors of Artin rings. \par
Based on the above discussions, the following Theorem was proved in \cite{GM2}.
      \begin{theorem}[{\cite[Theorem 4.8.]{GM2}}]\label{Main GM}
      Suppose $(L_1,[,]_{1},d_{L_{1}})$ and $(L_2,[,]_2,d_{L_2})$ are analytic DGLAs which are quasi-isomorphic as DGLAs. Then the analytic germ $(Kur_{L_1},0)$ and $(Kur_{L_2},0)$ are analytically isomorphic.
     \end{theorem}
\begin{remark}
The DGLAs that appear in this paper are analytic DGLAs by the standard Sobolev norms.
\end{remark}
\begin{remark}
The construction of $(Kur_L,0)$ is based on \cite{Ku}. When a DGLA $(L,[,],d_L)$ comes from a differential geometric object, the complement $C^1(L)$ is obtained by the Hodge decomposition of the differential $d_L$. In this case, $Kur_L$ is the standard versal deformation space. For example, when $(A^*(TX),[,]_{SN},\overline\partial_{TX})$ is the Kodaira-Spencer algebra of a complex manifold $X$, then $Kur_{A^*(TX)}$ is exactly the Kuranishi space of $X$.
\end{remark}
\section{Deformation of holomorphic-Higgs pairs}\label{DGLA}
In this section, we review our previous work \cite{O}.\par
Let $X$ be a complex manifold and $(E,\overline\partial_E,\theta)$ be a Higgs bundle over $X$.  We called the pair $(X, E,\theta)$ a \textit{holomorphic-Higgs pair}.  In our previous paper \cite{O}, we considered the deformation problem of holomorphic-Higgs pairs and constructed the DGLA that governs the deformation and constructed the Kuranishi space. 
We give the definition of the deformation of the holomorphic-Higgs pair $(X, E,\theta)$.
\begin{definition}
Let $(X,E,\theta)$ be a holomorphic-Higgs pair. A family of deformation of holomorphic-Higgs
pair over a small ball $\Delta$ centered at the origin of $\mathbb{C}^d$, a complex manifold $\mathcal{X}$, a proper holomorphic submersion 
\begin{equation*}
\pi:\mathcal{X}\to \Delta
\end{equation*}
and a Higgs bundle $(\mathcal{E},\Theta)$ such that, $\pi^{-1}(0)=X, \mathcal{E}|_{\pi^{-1}(0)}$, and $\Theta|_{\pi^{-1}(0)}=\theta$.
\end{definition}
\subsection{DGLA}
In this section, we introduce the DGLA which governs the deformation of holomorphic-Higgs pair based on \cite{O}. \par
Let $(X,E, \theta)$ be a holomorphic-Higgs pair and $TX$ be a holomorphic tangent bundle of $X$. We fix a hermitian metric $K$ on $E$. Let $\partial_K$ be a (1,0)-part of the Chern connection with respect to $\overline\partial_E$ and $K$. For $\phi\in A^{(0,i)}(TX)$ and $\partial_K$, we define
\begin{equation*}
\{\partial_K,\phi\lrcorner\}:=\partial_K(\phi\lrcorner)+(-1)^i\phi\lrcorner\partial_K.
\end{equation*}
Here, $\phi\lrcorner$ is the contraction with respect to $\phi$.\par
Let $L^i:=\bigoplus_{p+q=i}A^{p,q}(\mathrm{End}E)\bigoplus A^{0,i}(TX)$ and $L:=\bigoplus_iL^i$. Let $(A,\phi)\in L^i$ and $(B,\psi)\in L^j$. We define,
\begin{equation*}
[(A,\phi),(B,\psi)]_L:=((-1)^i\{\partial_K,\psi\lrcorner\}A-(-1)^{(i+1)j}\{\partial_K,\phi\lrcorner\}B-[A,B]_{\mathrm{End}E},[\phi,\psi]_{SN}).
\end{equation*}
Here, $[,]_{SN}$ is the standard Schouten-Nijenhuys bracket defined on $\bigoplus_iA^{0,i}(TX).$ \par
We define $B\in A^{0,1}(\mathrm{Hom}(TX,\mathrm{End}E))$ and a $\mathbb{C}$-linear map $C_i:A^{0,i}(TX)\to A^{1,i}(\mathrm{End}E)$ such that they acts on $\phi\in A^{0,i}(TX)$ as 
\begin{equation*}
B(\phi):=(-1)^i\phi\lrcorner F_{d_K}, C_i(\phi):=\{\partial_K,\phi\lrcorner\}\theta.
\end{equation*}
We define a $\mathbb{C}$-linear map $d_L:L^i\to L^{i+1}$ as 
\begin{equation*}
d_L:=
\begin{pmatrix}
\overline\partial_{\mathrm{End}E} & B \\
0 & \overline\partial_{TX} \\
\end{pmatrix}
+
\begin{pmatrix}
\theta & C_i \\
0 & 0 \\
\end{pmatrix}.
\end{equation*}
After some calculations, we obtain the following theorem.
\begin{theorem}[{\cite[Theorem 3.1.]{O}}]\label{Main DGLA}
$\big( L,[,]_L, d_L\big)$ is a DGLA. 
\end{theorem}
This DGLA governs the deformation of $(X, E,\theta)$. Actually, let $(A,\phi)\in L^1$. Then $(A,\phi)$ defines a holomorphic-Higgs pair if and only if $(A,\phi)$ is a Maurer-Cartan element. This was proved in \cite[Theorem 3.6.]{O}.
\subsection{Kuranishi Space}
We use the same notation as the previous section. In this section, we introduce the Kuranishi Space and the Kuranishi family for a given holomorphic-Higgs pair $(X, E,\theta)$.  Briefly, Kuranishi Space is an analytic space and the Kuranishi family is a family of holomorphic-Higgs pairs parametrized by Kuranishi Space such that every holomorphic-Higgs pair which comes from a small deformation of $(X,E,\theta)$ is isomorphic to a holomorphic-Higgs pair which is in the Kuranishi family.  \par
Since $(L,[,]_L,d_L)$ is constructed differential geometrically, we can apply the Kuranishi's work \cite{Ku} to construct the Kuranishi family. Let $d_L^*$ be the formal adjoint of $d_L$ with respect to $L^2$-inner product, $\Delta_L:=d_Ld_L^*+d_L^*d_L$ be the Laplacian and $G_L$ be the Green operator associated to $\Delta_L$. Let $\mathbb{H}^i:=\mathrm{ker} (\Delta_L:L^i\to L^i)$ and $H:L^i\to\mathbb{H}^i$ be the projection. By the classical Hodge theory, we know that $\mathrm{dim}\mathbb{H}^i$ has a finite dimension for each $i$. Let $\{\eta_i\}_{i=1}^n$ be an orthogonal bases of $\mathbb{H}^1$ with respect to $L^2$-inner product. For each $t=(t_1,\dots,t_n)\in\mathbb{C}^n$, we set $\epsilon_1(t):=\sum_it_i\eta_i$.
\begin{lemma}[{\cite[Proposition 4.1, Proposition 4.2.]{O}}]
For any $t\in\mathbb{C}^n$ and $|t|<<1$, there is a $\epsilon(t)\in L^1$ such that $\epsilon(t)$ satisfies the equation
\begin{equation*}
\epsilon(t)=\epsilon_1(t)+\frac{1}{2}d^*_{L}G_L[\epsilon(t),\epsilon(t)]_L
\end{equation*}
and $\epsilon(t)$ depends holomorphically with respect to variable $t$.\par
Moreover, $\epsilon(t)$ satisfies the Maurer-Cartan equation if and only if 
\begin{equation*}
H[\epsilon(t),\epsilon(t)]=0.
\end{equation*}
\end{lemma}
Let $\Delta\subset\mathbb{C}^n$ be a small ball such that $\epsilon(t)$ is holomorphic on $\Delta$. We set,
\begin{equation*}
Kur_{(X,E,\theta)}:=\{t\in\Delta : H[\epsilon(t),\epsilon(t)]=0\}.
\end{equation*}
Since the dimension of  $\mathbb{H}^2$ is finite, $Kur_{(X, E,\theta)}$ is an analytic space. We call $Kur_{(X, E,\theta)}$ the Kuranishi space of $(X,E,\theta)$. Since a Maurer-Cartan element defines a holomorphic-Higgs pair, we obtain a family of holomorphic-Higgs pair $\{(X_{\epsilon(t)},E_{\epsilon(t)},\theta_{\epsilon(t)})\}_{t\in Kur_{(X, E,\theta)}}$. We call this family the Kuranishi family of $(X,E,\theta)$. The Kuranishi space and the Kuranishi family contain all small deformation of $(X,E,\theta)$. Actually, let $|\cdot|_k$ be the $k$-th Sobolev norm of $L^1$ and let $\eta\in L^1$ be a Maurer-Cartan element. If $|\eta|_k<<1$, then there exists a $t\in Kur_{(X, E,\theta)}$ such that 
\begin{equation*}
(X_\eta, E_\eta, \theta_\eta)\simeq (X_{\epsilon(t)},E_{\epsilon(t)},\theta_{\epsilon(t)}).
\end{equation*}
Here $(X_\eta, E_\eta, \theta_\eta)$ is the holomorphic-Higgs pair which $\eta$ determines. This was proved in \cite[Theorem 4.2.]{O}.

\section{Harmonic bunldes}
\subsection{Harmonic bundles}\label{hb}
In this section, we assume $(X,\omega)$ to be a compact K\"ahler manifold. \par
Let $(E,\overline\partial_E,\theta)$ be a Higgs bundle over $X$ and  $K$ be a hermitian metric of $E$. Let $\partial_K$ be the (1,0)-part of the Chern connection associated with $\overline\partial_E$ and $K$ and $\theta_K^{\dagger}$ be the adjoint of $\theta$ with respect to $K$. We set $D'_K:=\partial_K+\theta_K^{\dagger}$ and $D'':=\overline\partial_E+\theta$. We obtain a connection $D:=D_K+D'$. We call a metric $K$ a harmonic metric when $D$ is a flat bundle (i.e. $D^2=0$). We call $(E,\overline\partial_E,\theta, K)$ a harmonic bundle when $K$ is a harmonic metric.\par
The existence of the harmonic metric on a Higgs bundle is related to the stability of it. The Riemann surface case was proved in \cite{N} and for the general K\"ahler case was proved in \cite{S1}.
\begin{theorem}[\cite{N,S1}]
Let $(E,\overline\partial_E,\theta)$ be a Higgs bundle on $X$. $(E,\overline\partial_E,\theta)$ admits a harmonic metric $K$ if and only if it is polystable and $c_1(E)=c_2(E)=0$. Let $K_1$ and $K_2$ be harmonic metrics. Then there exists a decomposition $(E,\overline\partial_E,\theta)=\oplus_i(E_i,\overline\partial_{E_i},\theta_i)$ such that (i) the decomposition is orthogonal with respect to both $K_1$ and $K_2$, (ii) there exists $a_i>0$ such that  $K_1=a_iK_2$ on $(E_i,\overline\partial_{E_i},\theta_i).$
\end{theorem}

 \subsection{K\"ahler Identities}
 We use the same notation as the previous section.\par
Let $A^p(E)$ be the space of the $p$ forms which takes value at $E$. We define a $L^2$-metric on $A^p(E)$ by using the Riemannian metric $g$ on $X$ and the Hermitian metric $K$ on $E$. Let $D^*$, $(D'_K)^*$ and $(D'')^*$ be the formal adjoint of $D$, $D'_K$ and $D''$ with respect to the $L^2$ inner product. Let $\Lambda_\omega$ be the contraction with respect to the Kahler form $\omega$. The following K\"ahler identities were proved in \cite[Lemma 3.1.]{S1}. 
\begin{lemma}
Let $(X,\omega)$ be a compact K\"ahler manifold, $(E,\overline\partial_E,\theta)$ be a Higgs bundle over $X$ and $K$ be a hermitian metric. Let $D'_K$,  $D''$, $(D'_K)^*$ and $(D'')^*$ be as above. Then the following equalities hold.
\begin{equation*}
(D'_K)^*=\sqrt{-1}[\ \Lambda_\omega, D'' ]\ , (D'')^*=-\sqrt{-1}[\ \Lambda_\omega, D'_K ]\ .
\end{equation*}
\end{lemma}
We define the laplacians as follows,
\begin{align*}
\Delta:=&DD^*+D^*D,\\
\Delta'':=&D''(D'')^*+(D'')^*D'',\\
\Delta'_K:=&D'_K(D'_K)^*+(D'_K)^*D'_K.
\end{align*}
We assume $K$ to be a harmonic metric. Under this assumption, we have $D''D'_K+D'_KD''=0$ and by the K\"ahler identities, we obtain the following equalities.
\begin{equation*}
\Delta= 2\Delta''= 2\Delta'_K.
\end{equation*}
Let $G, G''$ and  $G'_K$ be the Green operators asscoiated to $\Delta,\Delta''$ and $\Delta'_K$. By the above relations of laplacians, we have $ 2G=G'=G'_K.$
We set $\mathbb{H}^i:=\text{ker}\Delta$. By the classical Hodge theory, we have the following orthogonal decompositions with respect to the $L^2$-inner product.
\begin{align*}
    A^{i}(E)&=\mathbb{H}^i\oplus \text{im}D\oplus \text{im}D^*,\\
    A^{i}(E)&=\mathbb{H}^i\oplus \text{im}D''\oplus \text{im}(D'')^*,\\
    A^{i}(E)&=\mathbb{H}^i\oplus \text{im}D'_K\oplus \text{im}(D'_K)^*.
\end{align*}
The next lemma was proved in \cite{S2}. We call it $D'_KD''$-lemma in this paper.
\begin{lemma}[{\cite[Lemma 2.1.]{S2}}] 
    Let $(E, \overline\partial_E,\theta, K)$ be a harmonic bundle on $X$. Then
    \begin{equation*}
        \mathrm{ker}D'_K\cap\mathrm{ker}D''\cap(\mathrm{im}D'_K+\mathrm{im}D'')=\mathrm{im}D'_KD''.
    \end{equation*}
\end{lemma}
\begin{proof} We give the proof for convenience. This lemma was originally proved in \cite{S2}.\par
Let $\gamma\in A^i(E).$ Suppose $\gamma=D'_K\alpha+D''\beta$, $D'_K\gamma=0$ and $D''\gamma=0$. Let $\beta=\beta_0+D'_K\beta_1+(D'_K)^*\beta_2$ be the Hodge decomposition with respect to $D'_K$ with $\beta_0$ harmonic. Since $\Delta'_K=\Delta''$, $\Delta''\beta_0=0.$ Thus we have
\begin{equation*}
 D''\beta=D''D'_K\beta_1+D''(D'_K)^*\beta_2.
\end{equation*}
From the K\"ahler identities, we have $D''(D'_K)^*=\sqrt{-1}D''[\ \Lambda_\omega,D'']\ =\sqrt{-1}D''\Lambda_\omega D''=-\sqrt{-1}'[\ \Lambda_\omega,D'']\ D''=-(D'_K)^*D''.$ Hence we have 
\begin{equation*}
\gamma=D'_K\alpha+D''D'_K\beta_1-(D'_K)^*D''\beta_2.
\end{equation*}
Since $\gamma$ is $D'_K-$closed, $(D'_K)^*D''\beta_2$ is also. From the equation
\begin{equation*}
\big((D'_K)^*D''\beta_2,(D'_K)^*D''\beta_2\big)_{L^2}=\big(D''\beta_2,D'_K(D'_K)^*D''\beta_2\big)_{L^2}=0,
\end{equation*}
we obtain $(D'_K)^*D''\beta_2=0$ and therefore, $D''\beta=D''D'_K\beta_1.$ Here $( , )_{L^2}$ is the $L^2$-norm. We can show $D'_K\alpha=D''D'_K\alpha_1$ by using exactly the same argument as $\beta.$ Hence the claim is proved.
\end{proof}
\subsection{Formality}
We use the same notation as in the section \ref{hb}. \par
Let $(X,\omega)$ be a compact K\"ahler manifold, $(E,\overline\partial_E,\theta,K)$ be a harmonic bundle on $X$. We obtain three DG vector spaces $(A^*(E):=\oplus_iA^i(E),D)$, $(A^*(E),D'')$, and $(A^*(E),D'_K)$. We define $\mathbb{H}^i_{DR}$, $\mathbb{H}^i_{Dol}$, and $\mathbb{H}^i_{D'_K}$ to be the $i$-th cohomology of $(A^*(E),D)$, $(A^*(E),D'')$, and $(A^*(E),D'_K)$. These DG vector spaces satisfy formality conditions. 
\begin{lemma}[{\cite[P.83.]{GM1},\cite[Lemma 2.2.]{S2}}]\label{formal}
The natural morphisms induce quasi-isomorphisms of the following DG vector spaces.
\begin{align*}
(\mathrm{ker}D'_K,D'')&\to(A^*(E),D),\\
(\mathrm{ker}D'_K,D'')&\to(A^*(E),D''),\\
(\mathrm{ker}D'_K,D'')&\to(\mathbb{H}^*_{DR},0),\\
(\mathrm{ker}D'_K,D'')&\to(\mathbb{H}^*_{Dol},0)\\
(\mathrm{ker}D'_K,D'')&\to(\mathbb{H}^*_{D'_K},0).
\end{align*}
\end{lemma}  
\begin{proof}
We only prove the quasi-isomorphism of $i:(\mathrm{ker}D'_K,D'')\to(A^*(E),D)$.\par
$H^*(i)$ is surjective: Let $\alpha\in \mathrm{Ker}{D}$. We now consider $D'_K\alpha$ and show it is $D''$-closed. Since $K$ is a harmonic metric,  $D'_KD''+D''D'_K=0$. Therefore, $D''D'_K\alpha=-D'_KD''\alpha=D'_KD'_K\alpha=0$. The second equation follows from the assumption of $\alpha$. As $D'_K\alpha$ is $D'_K$-closed, we can apply the $D'_KD''$-lemma. Hence there exists a $\beta$ such that
\begin{equation*} 
D'_K\alpha=D'_KD''\beta.
\end{equation*}
 Moreover
 \begin{equation*}
  D''\alpha=-D'_KD''\beta.
  \end{equation*}
  Now we consider $\alpha-D\beta.$ From the equations
  \begin{align*}
  D'_K(\alpha-D\beta)&=D'_K\alpha-D'_KD''\beta=0,\\
  D''(\alpha-D\beta)&=D''\alpha-D''D'_K\beta=0,
    \end{align*}
    We have $\alpha-D\beta\in\mathrm{Ker}D'_K\cap\mathrm{Ker}D''$. Hence $\alpha-D\beta$ defines a cohomology class of $(\mathrm{ker}D'_K, D'')$. $H^*(i)$ maps the cohomology class of $\alpha-D\beta$ to the cohomology class of $\alpha$. Hence $H^*(i)$ is surjective.\par
    $H^*(i)$ is injective: Let $\alpha\in \mathrm{Ker}D'_K\cap\mathrm{Ker}D''$ and we assume that there exists a $\beta$ such that $\alpha=D\beta$. Under this assumption, We can apply the $D'_KD''$-lemma to $D'_K\beta$. Then there exists a $\gamma$ such that 
    \begin{equation*}
    D'_K\beta=D'_KD''\gamma.
       \end{equation*}
 Now we consider $\beta-D\gamma$. From the equations
 \begin{align*}
 D'_K(\beta-D\gamma)&=D'_K\beta-D'_KD''\gamma=0,\\
  D''(\beta-D\gamma)=D''\beta&-D''D'_K\gamma=D''\beta+D'_K\beta=D\beta=\alpha,
     \end{align*}      
     we obtain $\alpha\in D''(\mathrm{Ker}D'_K)$. Hence the cohomology class which $\alpha$ defines in $(\mathrm{ker}D'_K,D'')$ is 0. Therefore $H^*(i)$ is injective.
    \end{proof}
Let $E^*$ be the dual of $E$. For any $p,q\in\mathbb{Z}_{\ge0}$, $E^{*\otimes p}\otimes E^{\otimes q}$ has a induced harmonic metric from $E$. Hence, $(A^*(E^{*\otimes p}\otimes E^{\otimes q}), D)$ and $(A^*(E^{*\otimes p}\otimes E^{\otimes q}), D'')$ satisfy fomrality condition. We now focus on $p=q=1$ case. In this case, $E^*\otimes E=\mathrm{End}E$ and $A^*(\mathrm{End}E)$ has a naturally defined bracket $[,]_{\mathrm{End}E}$ such that for $A\in A^i(\mathrm{End}E)$ and $B\in A^j(\mathrm{End}E)$,
\begin{equation*}
[A,B]_{\mathrm{End}E}=A\wedge B-(-1)^{ij}B\wedge A.
\end{equation*}
By some calculation, we can show that $(A^*(\mathrm{End}E),[,]_{\mathrm{End}E},D)$ and $(A^*(\mathrm{End}E),[,]_{\mathrm{End}E},D'')$ are DGLA. Therefore, by Lemma \ref{formal}, we have the following lemma.
\begin{lemma}
$(A^*(\mathrm{End}E),[,]_{\mathrm{End}E},D)$ and $(A^*(\mathrm{End}E),[,]_{\mathrm{End}E},D'')$ are formal as DGLA.
\end{lemma}
 \section{Structure of Kuranishi space}
In this section, we study the structure of the analytic germ $(Kur_{(X,E,\theta)},0)$ when $(X,\omega)$ is a compact K\"ahler manifold and the Higgs bundle $(E,\overline\partial_E,\theta)$ on $X$ has a harmonic metric $ K.$ We prove that $(Kur_{(X,E,\theta)},0)\simeq (Kur_{(E,\theta)}\times Kur_{X},0)$ as analytic germs. We prove that by showing certain DGLAs are quasi-isomorphic and apply Theorem \ref{Main GM}.\par
Throughout this section, the DGLA $(L,[,]_L,d_L)$ is the DGLA in the Theroem \ref{Main DGLA}.
\subsection{DGLA}
In this section, we study the differential of $(L,[,]_L,d_L)$ when $(X,\omega)$ is a compact K\"ahler manifold and $(E,\overline\partial_E,\theta,K)$ is a harmonic bundle.
\begin{proposition}\label{do har}
When $X$ is a compact K\"ahler manifold and $(E,\overline\partial_E,\theta,K)$ is a harmonic bundle over $X$, then the differential $d_L$ of the DGLA $(L,[,]_L,d_L)$ acts on $(A,\phi)\in L^i$ as 
\begin{equation*}
d_L
\begin{pmatrix}
A\\
\phi
\end{pmatrix}
=
\begin{pmatrix}
D''A+D'_K(\phi\lrcorner \theta)\\
\overline\partial_{TX}\phi
\end{pmatrix}.
\end{equation*}
\begin{proof}
The second row is from the definition of $d_L$. From the definition of $d_L$, the first row of 
\begin{equation*}
d_L
\begin{pmatrix}
A\\
\phi
\end{pmatrix}
\end{equation*}
is 
\begin{equation*}
\overline\partial_{\mathrm{End}E}A+(-1)^i\phi\lrcorner F_{d_K}+[\theta, A]+\{\partial_K,\phi\lrcorner\}\theta.
\end{equation*}
Since $K$ is a harmonic metric, $D=D'_K+D''$ is flat. Therefore, the $(2,0)$-part and the $(1,1)$-part of $D^2$ is 0. The $(2,0)$-part is $\partial_K\theta$ and the $(1,1)$-part is $F_{d_K}+[\theta,\theta_K^\dagger]$. Hence we have the equality
\begin{align*}
\overline\partial_{\mathrm{End}E}A+(-1)^i\phi\lrcorner F_{d_K}+[\theta, A]+\{\partial_K,\phi\lrcorner\}\theta&=D''A-(-1)^i\phi\lrcorner[\theta,\theta_K^\dagger]+\partial_K(\phi\lrcorner\theta)\\
&=D''A+[\theta_K^\dagger,\phi\lrcorner\theta]+\partial_K(\phi\lrcorner\theta)\\
&=D''A+D'_K(\phi\lrcorner \theta).
\end{align*}
Hence the claim is proved.
\end{proof}
\end{proposition}
\subsection{Quasi-isomorphisms of DGLAs}
In this section, we prove quasi-isomorphisms of certain DGLAs. Let $(A^*(TX),[,]_{SN},\overline\partial_{TX})$ be the Kodaira-Spencer algebra. We first state the main result of this section.
\begin{theorem}\label{q iso}
$(L,[,]_L,d_L)$ is quasi-isomorphic to $(A^*(\mathrm{End}E),[,]_{\mathrm{End}E},D)\oplus(A^*(TX),[,]_{SN},\overline\partial_{TX})$.
\end{theorem}
We first prove the following Proposition.
\begin{proposition}\label{q iso1}
$(\mathrm{Ker}D'_K\oplus A^*(TX),[,]_L,d_L)$ is a sub DGLA of $(L,[,],d_L)$ and the canonical morphism $i:(\mathrm{Ker}D'_K\oplus A^*(TX),[,]_L,d_L)\to(L,[,]_L,d_L)$ is a quasi-isomorphism.
\end{proposition}
\begin{proof}
By the definition of $d_L$, it is easy to see that $(\mathrm{Ker}D'_K\oplus A^*(TX),d_L)$ is a sub DG vector space of $(L,d_L)$. We show that $[,]_L$ preserves $\mathrm{Ker}D'_K\oplus A^*(TX)$. Let $\alpha:=(A,\phi),\beta:=(B,\psi)\in\mathrm{Ker}D'_K\oplus A^*(TX)$. We have
\begin{align*}
[\alpha,\beta]_L&=
\begin{pmatrix}
(-1)^i\{\partial_K,\psi\lrcorner\}A-(-1)^{(i+1)j}\{\partial_K,\phi\lrcorner\}B-[A,B]\\
[\phi,\psi]_{SN}
\end{pmatrix}\\
&=
\begin{pmatrix}
(-1)^i\{\partial_K,\psi\lrcorner\}A-(-1)^{(i+1)j}\{\partial_K,\phi\lrcorner\}B-[A,B]\\
[\phi,\psi]_{SN}
\end{pmatrix}
\end{align*}
Since $A\in\mathrm{Ker}D'_K$, we have
\begin{align*}
\{\partial_K,\psi\lrcorner\}A&=\partial_K(\psi\lrcorner A)+(-1)^j\psi\lrcorner\partial_K A\\
&=\partial_K(\psi\lrcorner A)-(-1)^j\psi\lrcorner\theta^\dagger_KA\\
&=\partial_K(\psi\lrcorner A)+\theta^\dagger_K\psi\lrcorner A\\
&=D'_K(\psi\lrcorner A).
\end{align*}
Hence we have
\begin{equation}\label{bracket ker}
[\alpha,\beta]_L=
\begin{pmatrix}
(-1)^iD'_K(\psi\lrcorner A)-(-1)^{(i+1)j}D'_K(\phi\lrcorner B)-[A,B]\\
[\phi,\psi]_{SN}
\end{pmatrix}.
\end{equation}
Hence $[\alpha,\beta]_L\in \mathrm{Ker}D'_K\oplus A^*(TX)$. Therefore, $(\mathrm{Ker}D'_K\oplus A^*(TX),[,]_L,d_L)$ is a sub DGLA of $(L,[,]_L,d_L)$.
\par
 We show that the natural morphism $i:(\mathrm{Ker}D'_K\oplus A^*(TX),[,]_L,d_L)\to(L,[,],d_L)$ is a quasi-isomorphism.\par
$H^*(i)$ is surjective: Let $\eta:=(A,\phi)\in\mathrm{Ker}d_L$. We want to show that there exist
\begin{equation*}
 \eta'=(A',\phi')\in \bigg(\mathrm{Ker}D'_K\oplus A^*(TX)\bigg)\cap \mathrm{ker}d_L
 \end{equation*}
  and $\gamma\in L$ such that
  \begin{equation*} 
  \eta-\eta'=d_L\gamma.
  \end{equation*}
By Proposition \ref{do har}, we have
\begin{equation*}
d_L\eta=
\begin{pmatrix}
D''A+D'_K(\phi\lrcorner \theta)\\
\overline\partial_{TX}\phi
\end{pmatrix}
=0.
\end{equation*}
Let $\mathcal{A}$ be the harmonic projection of $A$ with respect to $D''$. The Hodge decomposition of $A$ with respect to $D''$ is 
\begin{align*}
A=\mathcal{A}+G''\Delta'' A&=\mathcal{A}+G''D''(D'')^*A+G''(D'')^*D''A\\
&=\mathcal{A}+D''(D'')^*G''A-\sqrt{-1}G''[\ \Lambda_\omega,D'_K]\ D''A\\
&=\mathcal{A}+D''(D'')^*G''A-\sqrt{-1}G''\Lambda_\omega D'_KD''A+\sqrt{-1}G''D'_K\Lambda_\omega D''A\\
&=\mathcal{A}+D''(D'')^*G''A+\sqrt{-1}G''\Lambda_\omega D'_KD'_K(\phi\lrcorner \theta)+\sqrt{-1}D'_KG''\Lambda_\omega D''A\\
&=\mathcal{A}+D''(D'')^*G''A+\sqrt{-1}D'_KG''\Lambda_\omega D''A.
\end{align*}
The compatibility of $D'_K$ and $G''$ follows from the fact that $G''=G'_K$.\par
We set
 \begin{align*}
 \eta':&=
 \begin{pmatrix}
 \mathcal{A}+\sqrt{-1}D'_KG''\Lambda_\omega D''A\\
  \phi
  \end{pmatrix},\\
   \gamma:&=
   \begin{pmatrix}
   (D'')^*G''A\\
   0
   \end{pmatrix}.
\end{align*}
Since $D''(\mathcal{A}+\sqrt{-1}D'_KG''\Lambda_\omega D''A)=\sqrt{-1}D''D'_KG''\Lambda_\omega D''A=D''A$ and $\mathcal{A}$ is the harmonic projection of $A$,
$\eta'\in\bigg(\mathrm{Ker}D'_K\oplus A^*(TX)\bigg)\cap \mathrm{ker}d_L$. By direct calculation, we can check $\eta-\eta'=d_L\gamma$. Hence, $H^*(i)$ is surjective.\par
$H^*(i)$ is injective: Let $\eta:=(A,\phi)\in\bigg(\mathrm{Ker}D'_K\oplus A^*(TX)\bigg)\cap \mathrm{ker}d_L$. We assume that there exsits a $\beta:=(B,\psi)\in L$ such that $\alpha=d_L\beta.$ Under this assumption, we have
\begin{equation*}
\begin{pmatrix}
A\\
\phi
\end{pmatrix}
=\alpha=d_L\beta=
\begin{pmatrix}
D''B+D'_K(\psi\lrcorner\theta)\\
\overline\partial_{TX}\psi
\end{pmatrix}.
\end{equation*}
Since $A\in\mathrm{ker}D'_K$, $D''B\in\mathrm{ker}D'_K\cap\mathrm{ker}D''\cap(\mathrm{im}D''+\mathrm{im}D'_K)$. Hence we can apply $D'_KD''$-lemma to $D''B$. Let $C\in A^*(\mathrm{End}E)$ such that $D''B=D''D'_KC$. We set,
\begin{equation*}
\gamma:=
\begin{pmatrix}
D'_KC\\
\psi
\end{pmatrix}.
\end{equation*}
Then $\gamma\in\mathrm{Ker}D'_K\oplus A^*(TX) $ and $\alpha=d_L\gamma$ stands. Hence we showed that $H^*(i)$ is injective.
\end{proof}
For $A\in A^*(E)$, let  $[A]_{D'_K}$ be the cohomology class in $\mathbb{H}^*_{D'_K}$. Let $Q:\mathrm{Ker}D'_K\to \mathbb{H}^*_{D'_K}$ be the $\mathbb{C}$-linear map such that $Q(A)=[A]_{D'_K}$.
\begin{proposition}\label{q iso2}
The morphism 
\begin{equation*}
\begin{pmatrix}
-Q&0\\
0&Id_{TX}
\end{pmatrix}:
(\mathrm{Ker}D'_K\oplus A^*(TX),[,]_L,d_L)\to\bigg(\mathbb{H}^*_{D'_K}\oplus A^*(TX),[,]_{\mathrm{End}E}\oplus [,]_{SN},
\begin{pmatrix}
0&0\\
0&\overline\partial_{TX}
\end{pmatrix}\bigg)
\end{equation*}
is a morphism of DGLA and it is a quasi-isomorphism.
\end{proposition}
\begin{proof}
We first show that 
$\begin{pmatrix}
-Q&0\\
0&Id_{TX}
\end{pmatrix}$
is a morphism of DG vector spaces. Let $\alpha:=(A,\phi)\in\big(\mathrm{Ker}D'_K\oplus A^*(TX)\big).$ We have,
\begin{align*}
\begin{pmatrix}
0&0\\
0&\overline\partial_{TX}
\end{pmatrix}
\begin{pmatrix}
-Q&0\\
0&Id_{TX}
\end{pmatrix}
\begin{pmatrix}
A\\
\phi
\end{pmatrix}
&=
\begin{pmatrix}
0&0\\
0&\overline\partial_{TX}
\end{pmatrix}
\begin{pmatrix}
[A]_{D'_K}\\
\phi
\end{pmatrix}
=\begin{pmatrix}
0\\
\overline\partial_{TX}\phi
\end{pmatrix},
\\
\begin{pmatrix}
-Q&0\\
0&Id_{TX}
\end{pmatrix}
\circ d_L
\begin{pmatrix}
A\\
\phi
\end{pmatrix}
&=
\begin{pmatrix}
-Q&0\\
0&Id_{TX}
\end{pmatrix}
\begin{pmatrix}
D''A+D'_K(\phi\lrcorner \theta)\\
\overline\partial_{TX}\phi
\end{pmatrix}\\
&=
\begin{pmatrix}
[D''A+D'_K(\phi\lrcorner \theta)]_{D'_K}\\
\overline\partial_{TX}\phi
\end{pmatrix}.
\end{align*}
Since $A\in\mathrm{Ker}D'_K$, we can apply  $D'_KD''$-lemma to $D''A$. Hence there is a $B\in A^*(\mathrm{End}E)$ such that $D''A=D'_KD''B.$ Therefore
\begin{equation*}
[D''A+D'_K(\phi\lrcorner \theta)]_{D'_K}=[D'_KD''B+D'_K(\phi\lrcorner \theta)]_{D'_K}=[D_K(D''B+\phi\lrcorner \theta)]_{D'_K}=0.
\end{equation*}
Hence 
$\begin{pmatrix}
-Q&0\\
0&Id_{TX}
\end{pmatrix}$
is a morphism of DG vector spaces. We next show that it is compatible with the brackets.
Let $\alpha:=(A,\phi), \beta:=(B,\psi)\in\mathrm{Ker}D'_K\oplus A^*(TX)$. By (\ref{bracket ker}), we have
\begin{align*}
[\alpha,\beta]_L=
\begin{pmatrix}
(-1)^iD'_K(\psi\lrcorner A)-(-1)^{(i+1)j}D'_K(\phi\lrcorner B)-[A,B]\\
[\phi,\psi]_{SN}
\end{pmatrix}.
\end{align*}
Hence we have
\begin{align*}
\begin{pmatrix}
-Q&0\\
0&Id_{TX}
\end{pmatrix}
[\alpha,\beta]_L
&=
\begin{pmatrix}
-Q&0\\
0&Id_{TX}
\end{pmatrix}
\begin{pmatrix}
(-1)^iD'_K(\psi\lrcorner A)-(-1)^{(i+1)j}D'_K(\phi\lrcorner B)-[A,B]\\
[\phi,\psi]_{SN}
\end{pmatrix}\\
&=
\begin{pmatrix}
-[(-1)^iD'_K(\psi\lrcorner A)-(-1)^{(i+1)j}D'_K(\phi\lrcorner B)-[A,B]_{\mathrm{End}E}]_{D'_K}\\
[\phi,\psi]_{SN}
\end{pmatrix}\\
&=
\begin{pmatrix}
[[A,B]_{\mathrm{End}E}]_{D'_K}\\
[\phi,\psi]_{SN}
\end{pmatrix}
=
\begin{pmatrix}
[[A]_{D'_K},[B]_{D'_K}]_{\mathrm{End}E}\\
[\phi,\psi]_{SN}
\end{pmatrix}.
\end{align*}
Hence $\begin{pmatrix}
-Q&0\\
0&Id_{TX}
\end{pmatrix}$ is a morphism of DGLA.\par
We next show that 
$\begin{pmatrix}
-Q&0\\
0&Id_{TX}
\end{pmatrix}$
is a quasi-isomorphism.\par
$H^*\bigg(\begin{pmatrix}
-Q&0\\
0&Id_{TX}
\end{pmatrix}\bigg)$ is surjective: Let $([A]_{D'_K},\phi)\in \bigg(\mathbb{H}_{D'_K}^*\oplus A^*(TX)\bigg)\cap \mathrm{Ker}
\begin{pmatrix}
0&0\\
0&\overline\partial_{TX}
\end{pmatrix}$. We first show that 
\begin{equation*}
-D''A+D'_K(\phi\lrcorner\theta)\in\mathrm{Ker}D''\cap\mathrm{Ker}D'_K\cap\mathrm{im}D.
\end{equation*} 
Since $A\in\mathrm{Ker}D'_K$, $-D''A+D'_K(\phi\lrcorner\theta)\in\mathrm{Ker}D'_K$. \par
Since $\overline\partial_{TX}\phi=0,$ we have
\begin{align*}
D''(\phi\lrcorner\theta)&=\overline\partial_{\mathrm{End}E}(\phi\lrcorner\theta)+[\theta,\phi\lrcorner\theta]_{\mathrm{End}E}\\
&=\overline\partial_{TX}\phi\lrcorner\theta-\phi\lrcorner\overline\partial_{\mathrm{End}E}\theta+\frac{1}{2}\phi\lrcorner[\theta,\theta]_{\mathrm{End}E}\\
&=0.
\end{align*}
Hence
\begin{align*}
D''(-D''A+D'_K(\phi\lrcorner\theta))=-D'_KD''(\phi\lrcorner\theta)=0.
\end{align*}
Moreover, we have
\begin{equation*}
D(-A+\phi\lrcorner\theta)=D''A+D'_K(\phi\lrcorner\theta).
\end{equation*}
Hence we proved $-D''A+D'_K(\phi\lrcorner\theta)\in\mathrm{Ker}D''\cap\mathrm{Ker}D'_K\cap\mathrm{im}D$. Therefore, we can apply $D'_KD''$-lemma to $-D''A+D'_K(\phi\lrcorner\theta)$. Hence there is a $B\in A^*(\mathrm{End}E)$ such that
\begin{equation*}
-D''A+D'_K(\phi\lrcorner\theta)=D''D'_KB.
\end{equation*}
Equivalently, we have
\begin{equation*}
-D''(A+D'_KB)+D'_K(\phi\lrcorner\theta)=0.
\end{equation*}
Therefore
\begin{align*}
\begin{pmatrix}
-A-D'_KB\\
\phi
\end{pmatrix}
\in\mathrm{Ker}D'_K\oplus A^*(TX)\cap \mathrm{Ker}d_L,
\end{align*}
and
\begin{equation*}
\begin{pmatrix}
-Q&0\\
0&Id_{TX}
\end{pmatrix}
\begin{pmatrix}
-A-D'_KB\\
\phi
\end{pmatrix}
=
\begin{pmatrix}
[A+D'_KB]_{D'_K}\\
\phi
\end{pmatrix}
=
\begin{pmatrix}
[A]_{D'_K}\\
\phi
\end{pmatrix}.
\end{equation*}
Hence $H^*\bigg( \begin{pmatrix}
-Q&0\\
0&Id_{TX}
\end{pmatrix}\bigg)$
is surjective.\par
 $H^*\bigg( \begin{pmatrix}
-Q&0\\
0&Id_{TX}
\end{pmatrix}\bigg)$
is injective: Let $(A,\phi)\in\bigg(\mathrm{Ker}D'_K\oplus A^*(TX)\bigg)\cap\mathrm{Ker}d_L$. We assume that the cohomology class of $([A]_{D'_K},\phi)$  in \bigg($\mathbb{H}^*_{D'_K}\oplus A^*(TX),
\begin{pmatrix}
0&0\\
0&\overline\partial_{TX}
\end{pmatrix}\bigg)$
is 0.
Hence there exsit a $B\in A^*(\mathrm{End}E)$ and a $\psi\in A^*(TX)$  such that
\begin{align*}
 A&=D'_K B,\\ 
 \phi&=\overline\partial_{TX}\psi.
 \end{align*}
 We show that 
 \begin{equation*}
 A-D'_K(\psi\lrcorner\theta)\in\mathrm{Ker}D''\cap\mathrm{Ker}D'_K\cap\mathrm{im}D'_K. 
 \end{equation*}
 Since $A=D'_KB$ and $A\in\mathrm{Ker}D'_K$, $A-D'_K(\psi\lrcorner\theta)\in\mathrm{Ker}D'_K\cap\mathrm{im}D'_K$. We also have
 \begin{align*}
 D''(A-D'_K(\psi\lrcorner\theta))&=D''A+D'_KD''(\psi\lrcorner\theta)\\
 &=D''A+D'_K(\overline\partial_{TX}\psi\lrcorner\theta+\frac{1}{2}\psi\lrcorner[\theta,\theta]_{\mathrm{End}E})\\
 &=D''A+D'_K(\phi\lrcorner\theta).
  \end{align*}
  Since $(A,\phi)\in\mathrm{Ker}d_L$, we have
  \begin{equation*}
  d_L
  \begin{pmatrix}
  A\\
  \phi
 \end{pmatrix}
 =
 \begin{pmatrix}
 D''A+D'_K(\phi\lrcorner\theta)\\
 \overline\partial_{TX}\phi
  \end{pmatrix}
  =0.
  \end{equation*}
  Therefore we have
  \begin{equation*}
  D''(A-D'_K(\psi\lrcorner\theta)= D''A+D'_K(\phi\lrcorner\theta)=0.
  \end{equation*}
  Hence we showed that $A-D'_K(\psi\lrcorner\theta)\in\mathrm{Ker}D''\cap\mathrm{Ker}D'_K\cap\mathrm{im}D'_K$. Therefore we can apply $D'_KD''$-lemma to $A-D'_K(\psi\lrcorner\theta)$. Hence there exists a $C\in A^*(\mathrm{End}E)$ such that
  \begin{equation*}
  A-D'_K(\psi\lrcorner\theta)=D''D'_KC.
   \end{equation*}
   We note that $(D'_KC,\psi)\in\mathrm{Ker}D'_K\oplus A^*(TX)$ and 
   \begin{equation*}
   d_L
   \begin{pmatrix}
   D'_KC\\
   \psi
  \end{pmatrix}
  =
  \begin{pmatrix}
  D''D'_KC+D'_K(\psi\lrcorner\theta)\\
  \overline\partial_{TX}\psi
   \end{pmatrix}
   =
   \begin{pmatrix}
   A\\
   \phi
   \end{pmatrix}.
   \end{equation*}
  Therefore the cohomology class of $(A,\phi)$ in $(\mathrm{Ker}D'_K\oplus A^*(TX),d_L)$ is 0. Hence  $H^*\bigg( \begin{pmatrix}
-Q&0\\
0&Id_{TX}
\end{pmatrix}\bigg)$
is injective.
 \end{proof}
\begin{proof}[Proof of Theorem \ref{q iso}]
Combining Lemma \ref{formal}, Proposition \ref{q iso1}, and Proposition \ref{q iso2} we have the following chain of DGLAs
\begin{align*}
(L,[,]_L,d_L)&\leftarrow(\mathrm{Ker}D'_K\oplus A^*(TX),[,]_L,d_L)\\
&\to\bigg(\mathbb{H}^*_{D'_K}\oplus A^*(TX),[,]_{\mathrm{End}E}\oplus [,]_{SN},
\begin{pmatrix}
0&0\\
0&\overline\partial_{TX}
\end{pmatrix}\bigg)\\
&\leftarrow \bigg(\mathrm{Ker}D'_K\oplus A^*(TX),[,]_{\mathrm{End}E}\oplus [,]_{SN},
\begin{pmatrix}
D&0\\
0&\overline\partial_{TX}
\end{pmatrix}\bigg)\\
&\to\bigg(A^*(\mathrm{End}E)\oplus A^*(TX),[,]_{\mathrm{End}E}\oplus [,]_{SN},
\begin{pmatrix}
D&0\\
0&\overline\partial_{TX}
\end{pmatrix}\bigg)
\end{align*}
such that each morphism is quasi-isomorphism. 
Hence the claim is proved.
\end{proof}
\begin{corollary}\label{q iso3}
$(L,[,]_L,d_L)$ is quasi-isomorphic to $(A^*(\mathrm{End}E),[,]_{\mathrm{End}E},D'')\oplus(A^*(TX),[,]_{SN},\overline\partial_{TX})$.
\end{corollary}
Let $Kur_X$ be the Kurainshi space of $X$, $Kur_{(E,\theta)}$ be the Kuranishi space of the Higgs bundle $(E,\theta)$, and $Kur_{(E,D)}$ be the Kuranishi space of the flat bundle $(E, D).$\par
Combining Theorem \ref{Main GM} and Theorem \ref{q iso}, we have the following theorem.
\begin{theorem}\label{main}
Let $(X,\omega)$ be a compact K\"ahler manifold, $(E,\overline\partial_E,\theta)$ be a Higgs bundle over $X$ and, $K$ be a harmonic metric. Then
\begin{align*}
(Kur_{(X,E,\theta)},0)&\simeq(Kur_{(E,\theta)}\times Kur_X,0),\\
(Kur_{(X,E,\theta)},0)&\simeq(Kur_{(E,D)}\times Kur_X,0)
\end{align*}
holds as germs of analytic spaces.
\end{theorem}
Theorem \ref{main} predicts that once we construct a moduli space of a pair of a compact K\"ahler manifold and a polystable Higgs bundle with vanishing Chern classes,  such moduli space should locally decompose to the direct product of the Kuranishi space of the manifold and the Kuranishi space of the Higgs bundle which we cannot expect globally.\par
We have some consequences from Theorem \ref{main} for specific cases. Let $M$ be a Riemann surface with genus $g \ge 2$ and $(E,\overline\partial_E,\theta)$ be a stable Higgs bundle of degree 0. Under these assumptions, each deformations of $M$ and $(E,\overline\partial_E,\theta)$ are unobstruced. Hence $Kur_M$ and $Kur_{(E,\theta)}$ are complex manifolds. Moreover, the dimensions of  $Kur_X$ is  $3g-3$ and $Kur_{(E,\theta)}$  is $2+r^2(2g-2)$ \cite{MK, N}. Here $r$ is the rank of $E$. The following is straightforward from Theorem \ref{main}.
\begin{corollary}\label{sub}
Let $M$ be a Riemann surface with genus $g \ge 2$ and $(E,\overline\partial_E,\theta)$ be a stable Higgs bundle of degree 0. Then the deformation of pair $(M, E,\theta)$ is unobstructed. Moreover, $Kur_{(M,E,\theta)}$ is a complex manifold and its dimension is $g(2r^2+3)-2r^2-1$.
\end{corollary}
The Corollary predicts that the moduli space of a pair of Riemann surfaces and a stable Higgs bundle of degree 0 is smooth in a stable locus and its dimension is $g(2r^2+3)-2r^2-1$.

\end{document}